\documentclass[12pt]{amsart}

\usepackage{amssymb, amsmath, amsthm,stmaryrd}
\usepackage[backref]{hyperref}
\usepackage[alphabetic,backrefs,lite]{amsrefs}
\usepackage[utf8]{inputenc} 

\usepackage{amscd}   
\usepackage{fullpage}
\usepackage[all]{xy} 
\usepackage{MnSymbol,graphicx}
\usepackage{mathrsfs}

\usepackage{mathtools}

\usepackage{makecell}

\usepackage{tikz}
\usepackage{verbatim}

  \tikzstyle{block} = [rectangle, draw,
      text width=8em, text centered, minimum height=4em]
  \tikzstyle{line} = [draw, -latex']

\usepackage{amssymb, amsmath, amsthm,stmaryrd}
\usepackage[backref]{hyperref}
\usepackage[alphabetic,backrefs,lite]{amsrefs}
\usepackage{amscd}   
\usepackage{fullpage}
\usepackage[all]{xy} 
\usepackage{MnSymbol,graphicx}
\usepackage{mathrsfs,enumitem}
\usepackage{multirow,qtree}
\DeclareFontEncoding{OT2}{}{} 

\usepackage{tikz}

\newtheorem{lemma}{Lemma}[section]
\newtheorem{theorem}[lemma]{Theorem}
\newtheorem{proposition}[lemma]{Proposition}

\newtheorem{cor}[lemma]{Corollary}
\newtheorem{conj}[lemma]{Conjecture}

\newtheorem{claim*}{Claim}

\theoremstyle{definition}
\newtheorem{remark}[lemma]{Remark}

\newcommand{\F}{{\mathbb F}}

\newcommand{\Q}{{\mathbb Q}}

\newcommand{\calN}{{\mathcal N}}
\newcommand{\calO}{{\mathcal O}}

\newcommand{\frakf}{{\mathfrak f}}
\newcommand{\frakfp}{{\mathfrak f'}}

\newcommand{\frakp}{{\mathfrak p}}
\newcommand{\frakq}{{\mathfrak q}}

\newcommand{\frakD}{{\mathfrak D}}

\newcommand{\frakN}{{\mathfrak N}}

\DeclareMathOperator{\Frob}{Frob}

\DeclareMathOperator{\Aut}{Aut}
\DeclareMathOperator{\Gal}{Gal}

\DeclareMathOperator{\Norm}{ Norm}

\DeclareMathOperator{\SL}{SL}
\DeclareMathOperator{\GL}{GL}

\DeclareMathOperator{\nr}{Norm}

\numberwithin{equation}{section}
\numberwithin{table}{section}

\title{	non-trivial Solutions of $Aa^p+Bb^p=Cc^3$ over Number Fields}

\author{Yasemin Kara}
\author{Stef Nomden}
\author{Ekin \"Ozman}

\address{Bogazici University\\
Department of Mathematics\\
Bebek, 34342 \\
Istanbul, Turkiye.}

\email{yasemin.kara@bogazici.edu.tr}

\address{Bernoulli Institute \\
Nijenborgh 9, 9747 AG \\
Groningen, the Netherlands.}

\email{stef.nomden@gmail.com}

\address{Bernoulli Institute \\
Nijenborgh 9, 9747 AG \\
Groningen, the Netherlands.}

\email{e.ozman@rug.nl}


\begin{document}

	\maketitle
	\begin{abstract}

	In this paper, we investigate solutions to the Diophantine equation
\[
A a^p + B b^p = C c^3
\]
over number fields using the modular method. Assuming certain standard modularity conjectures, we first establish an asymptotic result for general number fields satisfying an appropriate $S$-unit condition. In particular, we verify that this condition holds for several imaginary quadratic fields.

\medskip
Beyond the asymptotic setting, we also obtain an effective result. Specifically, for the equation
\[
a^p + d b^p = c^3
\]
over \( K = \mathbb{Q}(\sqrt{-d}) \) with \( d \in \{7, 19, 43, 67\} \), we determine an explicit bound (depending on \( d \)) such that no non-trivial solutions of a certain type exist whenever \( p \) exceeds this bound.
\end{abstract}

	\section{Introduction}
	
    The proof of Fermat's Last Theorem in 1995 was a monumental achievement in mathematics. Not only was this the end of the famous open problem that remained for more than 300 years, but it also paved the way to a new strategy of solving certain Diophantine equations. This strategy is called the `modular method' which exploits the relation between modular forms and elliptic curves. Fermat's Last Theorem states that there are no non-trivial ($abc \neq 0$) solutions to the equation $a^p + b^p = c^p$ whenever $p>2$. Assuming the existence of such a solution, Frey constructed a semistable elliptic curve to this solution in his paper \cite{Frey}. In 1990, Ribet showed in \cite{Ribet} that if this elliptic curve is modular, then there is a Galois representation of level two which is modular. This would give a contradiction as there are no newforms of level two. In 1995, Wiles proved in \cite{Wiles} that semistable elliptic curves are modular which was the final piece to prove Fermat's Last Theorem.
    
    This connection between elliptic curves and modular forms is also conjectured over general number fields. Assuming these conjectures, the modular method can be employed to solve other Diophantine equations over number fields. The modular method was used by Freitas and Siksek in \cite{FS} to show that the equation $a^p + b^p = c^p$ has no `asymptotic solutions' over a totally real number field $K$. That is, for a number field $K$, there is a constant $B_K$ depending on $K$ such that whenever $p> B_K$, there is no non-trivial solution to the equation. In \cite{HD}, Deconinck generalized these methods to show that the equation $Aa^p + Bb^p = Cc^p$ has no asymptotic solutions over a totally real field $K$ with $A,B$ and $C$ odd. The modular method also applies to asymptotically solve $a^p + b^p = c^3$ over certain number fields $K$; in \cite{IKO2}, Isik, Kara and Özman showed that there are no asymptotic solutions to this equation over particular number fields $K$. Furthermore, in \cite{moc22} Mocanu proved that the same equation has no asymptotic solutions for totally real number fields $K$ which satisfy a certain condition about $S$-units.

 In this paper, we combine and generalize the approaches of  \cite{IKO2} and \cite{moc22} to prove that the equation $Aa^p + Bb^p = Cc^3$ admits no asymptotic solutions for number fields $K$ satisfying a suitable $S$-unit condition.  In particular, we show that for certain imaginary quadratic fields, this condition is satisfied.

    In addition to this asymptotic result, we also establish an effective result that extends the explicit bounds obtained in \cite{IKO2}. More precisely, we show that for the equation 
\[
a^p + d b^p = c^3
\]
over $K = \mathbb{Q}(\sqrt{-d})$, with $d \in \{7,19,43,67\}$, there exists an explicit bound (depending on $d$) such that there are no non-trivial solutions of a particular type whenever $p$ exceeds this bound.

     \subsection*{Our Results} 
Let $K$ be a number field and let $\mathcal{O}_K$ denote its ring of integers and $\mathcal{O}_K^\times$ denote the group of units in $\mathcal{O}_K$. For $A,B,C$  nonzero elements of $\calO_{K}$
and a prime number $p$, we refer the equation
	  \begin{equation}\label{maineqn}
	  Aa^p+Bb^p=Cc^3
	  \end{equation}
	  as \textit{the generalized Fermat equation over K with coefficients A,B,C and signature (p,p,3)}.  A solution $(a,b,c)$ is
called \textbf{ trivial} if $abc=0$, otherwise \textbf{non-trivial}. A solution $(a,b,c)$ of Equation \ref{maineqn} is called \textbf{primitive} if the ideal generated by $a$, $b$ and $c$ is equal to the unit ideal.

We study solutions $(a,b,c)\in\mathcal{O}_K^3$ to \eqref{maineqn} such that $(a,b,c)$ is primitive. This restriction of solutions has the following reason. We aim to show that there are no non-trivial solutions to \eqref{maineqn} for $p$ larger than a constant depending only on $K,A,B$ and $C$. If we do not require a solution to be primitive, then we obtain an infinite set of uninteresting solutions to \eqref{maineqn}.
In fact, if $p\equiv -1 \pmod{3}$ and $a,b,c \in \mathcal{O}_K$ satisfy $Aa^p+Bb^p=Cc$, then $(ac,bc,c^{\frac{p+1}{3}})$ is a non-primitive solution to $\eqref{maineqn}$. Observe that if $x,y,z \in \mathcal{O}_K$ satisfy $Ax^p+By^p=Cz$, then $(a,b,c)=(xz,yz,z^{\frac{p+1}{2}}) $ satisfies the equation $Aa^p+Bb^p=Cc^2$. Hence, if $p\equiv 1 \pmod{3}$ we find that $(ac,bc,c^{\frac{p+2}{3}})$ is a non-primitive solution to  $\eqref{maineqn}$. Thus, we consider only primitive solutions to \eqref{maineqn}. In \cite{DG}, it was shown that Equation \ref{maineqn} has finitely many primitive solutions.  \\

The nature of our results to \eqref{maineqn} is asymptotic and explicit.  The main asymptotic result  is about some type of solutions over general number fields, and the explicit result is about solutions over some imaginary quadratic fields. Before stating the main results, we introduce some notation. Define $$T_K = \{\mathfrak{P} : \mathfrak{P} \text{ is a prime in } K \text{ and }\mathfrak{P}\mid 3ABC\}.$$ Let $\mathcal{O}_{T_K}$ denote the ring of $T_K$-integers and let $\mathcal{O}_{T_K}^\times$ denote the group of $T_K$-units. Further, let $\text{Cl}_{T_K}(K)$ denote the quotient $\text{Cl}(K)/\langle[\mathfrak{P}]\rangle_{\mathfrak{P}\in T_K}$ and $\text{Cl}_{T_K}(K)[n]$
denote its $n$-torsion. Let $W_K$ be the set $(a,b,c)\in\mathcal{O}_K^3$ such that $(a,b,c)$ is a primitive non-trivial solution to \eqref{maineqn} and every prime in $K$ above $3$ divides $ab$. Our asymptotic result is as follows. 

\begin{theorem}\label{main theorem}
    Let $K$ be a number field with $\textup{Cl}_{T_K}(K)[3]=1$ satisfying Conjectures \ref{conj1} and \ref{conj2}. Suppose further that for every solution $(\alpha,\beta,\gamma) \in\mathcal{O}_{T_K}^\times\times\mathcal{O}_{T_K}^\times\times\mathcal{O}_{T_K}$ to $\alpha+\beta=\gamma^3$, there exists a prime $\mathfrak{L}$ in $K$ above $3$ such that $$|v_\mathfrak{L}(\alpha\beta^{-1})|\leqslant 3v_\mathfrak{L}(3).$$
    Then, there is a constant $V=V(K,A,B,C)>0$ depending only on $K$, $A$, $B$ and $C$ such that  Equation \ref{maineqn} has no non-trivial solutions in $W_K$ for $p>V$. 
\end{theorem}
Theorem \ref{main theorem} generalizes Theorem 2.5 of \cite{Kumar} in the sense that \cite{Kumar} restricts to totally real number fields. What allows this generality in our case is that we use a slightly modified level lowering technique which is used in \cite[Section 3]{IKO2}. A slight downside of this method is that we need to assume two conjectures instead of one. In the totally real case, only one conjecture is assumed due to the level lowering result \cite[Theorem 7]{FS} of Freitas and Siksek for elliptic curves. The restriction of $K$ regarding $T_K$-units in the assumptions of Theorem \ref{main theorem} will be relevant in Section \ref{T-unit section}. The method employed in that section is due to Mocanu in \cite{moc22} and is also used in \cite{Kumar}. In this paper, it is shown that there is a finite number of solutions to $\alpha+\beta=\gamma^3$ with $(\alpha,\beta)\in\mathcal{O}_{T_K}^\times$ and $\gamma\in\mathcal{O}_{T_K}$ up to scaling by cubes. Further, it is shown that this set of solutions is effectively computable. As a consequence, for a given number field $K$, it can be determined if $K$ satisfies the hypothesis of Theorem \ref{main theorem} in finite time.  

Furthermore, by imposing some local constraints, we obtain the following theorem.

\begin{theorem} \label{theorem: special case}
    Let $K$ be a number field with only one prime $\mathfrak{q}$ above $3$ satisfying Conjectures \ref{conj1} and \ref{conj2}. Suppose that $\mathfrak{q}$ is principal and that $3\nmid h_{K}h_{K(\zeta_3)}$. Let $A,B$ and $C$ be elements in $\mathcal{O}_K$ supported only on $\mathfrak{q}$. Let $S_K = \{\mathfrak{q}\}$ and suppose that every solution $(\alpha,\gamma)\in\mathcal{O}_{S_K}^\times \times\mathcal{O}_{S_K}$ to 
    \begin{equation}
        \alpha + 1 = \gamma^3 \label{alphaplusoneequalsgammacubed}
    \end{equation}
    with $v_\mathfrak{q}(\alpha)\geqslant 0$ satisfies $v_\mathfrak{q}(\alpha) \leqslant 3v_\mathfrak{q}(3)$. Then, there is a constant $V = V(K,A,B,C)$ such that the equation $Aa^p + Bb^p = Cc^3$ with exponent $p>V$ has no non-trivial
    solutions in $W_K$.
\end{theorem}

In particular, when we apply the above theorem to imaginary quadratic fields, we obtain the following result:

\begin{cor}\label{corollary: quadratic imaginary fields asymptotic solutions}
    Let $d\geqslant 2$ be a square-free integer such that $d\equiv 1\mod 3$ and let $K = \mathbb{Q}(\sqrt{-d})$. Let $A$, $B$ and $C$ be elements of $\mathcal{O}_K$ divisible only by the primes above $3$ in $K$. Suppose that $3\nmid h_Kh_{K(\zeta_3)}$ and that $K$ satisfies Conjectures \ref{conj1} and \ref{conj2}. Then, there is a constant $V = V(d,A,B,C)$ such that the equation $Aa^p + Bb^p = Cc^3$ has no non-trivial solutions in $W_K$ when $p>V$. 
\end{cor}
	For explicit results, we restrict ourselves to the case $A=C=1$ and $B=d$ when working over the fields $K=\Q(\sqrt{-d}).$ There are of course many other equations for which the methods may work but we have chosen these cases to illustrate the method and the difficulties that may arise.

     \begin{theorem}\label{thm:sample}
         Let $K=\Q(\sqrt{-d})$ where $d \in \{7,19,43,67\}$.  Assume that Conjecture~\ref{conj1} holds for $K$. Note that $3$ is inert in all of these number fields and let $\lambda$ be the unique prime of $\calO_K$ lying over $3$. Then for any prime $p>B_K$, the equation
          \begin{equation}\label{maineqn1}
	  a^p+d b^p=c^3
	  \end{equation} does not have any non-trivial solutions $(a,b,c)\in\calO_{K}^3$ such that $(a,db,c)$ is primitive and $\lambda \mid b$ where $B_K$ is as follows:
      \begin{itemize}
          \item When $d=7,19$, $B_K=20.$
          \item When $d=43$, $B_K= 2531.$
          \item When $d=67$, $B_K= 86338229.$
      \end{itemize}
         \end{theorem}

       In some cases, it may be possible to get rid of the condition on the type of the solutions as in the following example.
 \begin{remark}\label{rk:sample2} In some cases, we were able to get a more general result. For instance,
         let $K=\Q(\sqrt{-43})$ and  assume that Conjecture~\ref{conj1} holds for $K$. Let $\lambda$ be the prime of $\calO_K$ lying over $3$. In this case, we are able to prove that if $p> B_K=44483={\rm max}\{\ell_K,C_K,M_K\}$, $p$ splits in $K$ and $p \equiv 3 \pmod 4$, then $a^p+43b^p=c^3$
	   does not have any non-trivial solutions $(a,b,c)\in\calO_{K}^3$ such that $(a,db,c)$ is primitive. Here, the notations $\ell_K,C_K,M_K$ are analogues of the constants in \cite{IKO2} Proposition 3.9 and Corollay 4.10. Analogous results hold in this case as well; however, since the proofs would require considerable length, we omit them here and record the results merely as a remark rather than as a theorem. Note that, this result requires elimination of newforms at a much higher level. This kind of elimination was unfortunately not possible for other fields.
         \end{remark}

     \subsection*{Aknowledgements} 
    We would like to thank Begum Gulsah Cakti for very helpful conversations. The first author is supported by the Turkish National and Scientific Research Council
(TÜBİTAK) Research Grant 122F413.

	  	\section{Preliminaries}
	
	 In this section, we give the necessary background to prove the results. We follow  \cite{SS}, \cite{FS} and the references therein.
	 \subsection{Conjectures}\label{Conjectures}
	 
	 In order to run the modular approach to solve Diophantine equations, one needs generalized modularity theorems.
	 Due to the lack of their existence, we can only prove our theorems up to some conjectures. These conjectures are related to eigenforms over a number field $K$ and mod $p$ eigenforms over $K$. A brief recap of these concepts is given in Sections 2 and 3 of \cite{SS}. 
	 
	 \begin{conj}[\cite{FKS}, Conjecture 4.1]\label{conj1}  Let $\overline{\rho}:G_K\rightarrow GL_2(\overline{\mathbb{F}}_p)$ be an odd, 
	 	irreducible, continuous representation with 
	 	Serre conductor $\frakN$ (prime-to-p part of its Artin conductor) 
	 	and such that $\det(\overline{\rho})=\chi_p$ is the mod p cyclotomic character.   
	 	Assume that $p$ is unramified in $K$ and that $\overline{\rho}|_{G_{K_{\frakp}}}$ arises from
	 	a finite-flat group scheme over $\calO_{K_{\frakp}}$ for every prime $\frakp\mid p$.  Then there is a weight two, mod $p$
	 	eigenform $\theta$ over $K$ of level $\frakN$ such that for all primes $\frakq$ coprime to $p\frakN$, we have
	 	\[
	 	\text{Tr}(\overline{\rho}(\Frob_{\frakq}))=\theta(T_{\frakq}),
	 	\]
	 	where $T_{\frakq}$ denotes the Hecke operator at $\frakq$.
	 \end{conj}
 
 The above conjecture follows from Serre's Modularity Conjecture. Additionally, we will use a special case of a fundamental conjecture from Langlands Programme for the asymptotic result.  Note that we don't need Conjecture \ref{conj2} for Theorem \ref{thm:sample}.
 
 \begin{conj}[\cite{SS}, Conjecture 4.1]\label{conj2}
 	Let $\frakf$ be a weight two complex eigenform over $K$ of level $\frakN$ that is non-trivial and new.  If $K$ has some
 	real place, then there exists an elliptic curve $E_{\frakf}/K$ of conductor $\frakN$ such that 
 	\begin{equation}\label{c2eqn}
 	\#E_{\frakf}(\calO_K/ \frakq)=1+\nr(\frakq)-\frakf(T_{\frakq})\quad\mbox{for all}\quad\frakq\;\nmid\;\frakN.
 	\end{equation}
 	If $K$ is totally complex, then there exists either an elliptic curve $E_{\frakf}$ of conductor $\frakN$ satisfying (\ref{c2eqn})
 	or a fake elliptic curve 
 	$A_{\frakf}/K$, of conductor $\frakN^2$, such that
 	\begin{equation}
 	\#A_{\frakf}(\calO_K/\frakq)=(1+\nr(\frakq)-\frakf(T_{\frakq}))^2\quad\mbox{for all}\quad\frakq\;\nmid\;\frakN.
 	\end{equation}
 \end{conj}
 
  Given a number field $K$, we obtain a \emph{complex conjugation} for every real embedding 
  $\sigma: K \hookrightarrow \mathbb R$ and every extension $\tilde{\sigma}: \overline{K} \hookrightarrow \mathbb C$ of 
  $\sigma$ as $\tilde{\sigma}^{-1}\iota \tilde{\sigma} \in G_K$ where $\iota$ is the usual complex conjugation.
  Recall that a 
  representation $\overline{\rho}_{E,p}: G_K \rightarrow \GL_2(\overline{\mathbb F}_p)$ is \emph{odd} if the determinant of every complex 
  conjugation is $-1$.  If the number 
  field $K$ has no real embeddings, then we immediately say that $\overline{\rho}_{E,p}$ is odd. 
  
	 \subsection{Frey curve and related facts}
	 
	 In this section, we collect some facts related to the Frey curve associated to a putative solution of Equation 
	 \ref{maineqn} and the associated Galois representation.
	 
	 Let $G_{K} $ be the absolute Galois group of a number field $K$, let $E/K$ be an elliptic curve and $ \overline{\rho}_{E,p}$ denote the mod $p$ Galois representation of $E$.  We use $\frakq$ for an arbitrary prime of $K$, and $G_\frakq$ and $I_\frakq$ respectively for the decomposition and inertia subgroups of $G_K$ at $\frakq$. 
	 
	  For a putative solution $(a,b,c)$ to the equation $Aa^p+Bb^p=Cc^3$ with a prime exponent $p$, we associate the Frey elliptic curve,
	 \begin{equation} \label{the frey curve}  
	 E=E_{a,b,c}: Y^2+3CcXY+C^2Bb^pY=X^3
	 \end{equation}
	 whose arithmetic invariants are given by $$\Delta_E=3^3AB^3C^8(ab^3)^p, \; \displaystyle j_E=\frac{3^3Cc^3(9Aa^p+Bb^p)^3}{AB^3(ab^3)^p}$$ and $$c_4(E)=3^2C^3c(9Aa^p+Bb^p),\; c_6(E)=-3^3C^4(3^3C^2c^6-2^23^2Cc^3Bb^p+2^3B^2b^{2p}).$$

	   Notice that the Frey curve $E$ has a $K$-rational point of order $3$ which is $(0,0)$.  Moreover, the determinant of $\overline{\rho}_{E,p}$ is the mod $p$ cyclotomic character which is  a well known consequence of the Weil pairing attached to elliptic curves.  This immediately implies oddness of $\overline{\rho}_{E,p}$ when we take $p$ large enough such that no $p^\text{th}$ root of unity is contained in $K$.

  In the next proposition, we discuss the semistability of the Frey curve $E$. 
\begin{proposition}\label{Conductor of frey}
    Let $E$ be the Frey curve defined in \eqref{the frey curve}. Then:
    \begin{enumerate}
    \item The Frey curve $E$ is semistable outside of $T_K$. 
    \item The Serre conductor $\mathfrak{N}_E$, which is the prime-to-$p$ part of the Artin conductor of $\overline{\rho}_{E,p}$,  is supported on the primes in $T_K$ and there is a finite amount of possible values for $\mathfrak{N}_E$. 
    \item For large enough $p$ (depending only on $A$,$B$,$C$ and $K$), the Galois representation $\overline{\rho}_{E,p}$ is finite flat at every prime $\mathfrak{p}$ of $K$ that lies above $p$.
    \end{enumerate}
\end{proposition}
\begin{proof}
  Let $\mathfrak{q}$ be a prime in $K$ with $\mathfrak{q}\not\in T_K$. If $\mathfrak{q}\nmid\Delta$, then $E$ has good reduction at $\mathfrak{q}$ \cite[Proposition VII.5.1]{Silverman86}. If $\mathfrak{q}\mid\Delta$, then $\mathfrak{q}\mid ab$ and since $(a,b,c)$ is primitive, it follows that $\mathfrak{q}$ can only divide one of $a$ or $b$. Indeed, if $\mathfrak{q}\mid a$ and $\mathfrak{q} \mid b$ then $\mathfrak{q}\mid Cc^3 = Aa^p + Bb^p$. Since $\mathfrak{q}\nmid C$ (as $\mathfrak{q}\not\in T_K$) it follows that $\mathfrak{q}\mid c$, a contradiction with primitivity.
    
    Since $\mathfrak{q}\not\in T_K$, it follows that $\mathfrak{q}\nmid 3^2C^3c(3^2Aa^p + Bb^p) = c_4$. Indeed, since $\mathfrak{q}$ can only divide one of $a$ or $b$ it follows that it cannot divide $c$ due to $(a,b,c)$ being primitive. Further, $3^2Aa^p + Bb^p$ is congruent to either $3^2Aa^p$ or $Bb^p$ if $\mathfrak{q}\mid b$ or $\mathfrak{q}\mid a$, respectively. In either case, we have $\mathfrak{q}\nmid 3^2Aa^p+Bb^p$. Combining these facts gives $\mathfrak{q}\nmid c_4$. It follows that $E$ is minimal at $\mathfrak{q}$ and $E$ has multiplicative reduction by \cite[Proposition VII.5.1]{Silverman86}.
    For the assertion on the Serre conductor $\mathfrak{N}_E$, note that apriori the Serre conductor is supported on the prime divisors of $3ABCab$. If $\mathfrak{q}\mid ab$, we have  $p\mid pv_\mathfrak{q}(ab^3) = v_\mathfrak{q}(\Delta)$.  It then follows from \cite{ser87} that $\overline{\rho}_{E,p}$ is unramified at $\mathfrak{q}$ and hence $\mathfrak{q}\nmid\mathfrak{N}_E$. Thus, for all primes $\mathfrak{q}\not\in T_K$, we have  $\mathfrak{q}\nmid\mathfrak{N}_E$. In other words, the Serre conductor $\mathfrak{N}_E$ is supported on the primes in $T_K$. Let $\mathcal{N}_E$ denote the conductor of $E$, then from \cite[Theorem IV.10.4]{advanced_silverman} it follows that for primes $\mathfrak{P}\in T_K$, $$v_\mathfrak{P}(\mathfrak{N}_E) \leqslant v_\mathfrak{P}(\mathcal{N}_E) \leqslant 2 + 3v_\mathfrak{P}(3) + 6v_\mathfrak{P}(2).$$ Hence, the Serre conductor $\mathfrak{N}_E$ can only be one of finitely many values. 
    Finally, let $\frakp$ be a prime of $K$ lying above $p>3$. If we assume $p \nmid ABC $, we see that $p\mid v_\frakp(\Delta_E)$. It then follows from \cite{ser87} that $\overline{\rho}_{E,p}$ is finite-flat at $\frakp$. 
\end{proof}

For the proof of  Theorem \ref{thm:sample}, we need to calculate the conductor of the Frey curve given in \eqref{the frey curve} explicitly when $C=1$ for the complex quadratic fields stated in the theorem.
	 In this case, the Frey elliptic curve becomes
	 \[
	 E=E_{a,b,c}: Y^2+3cXY+Bb^pY=X^3
	 \]
	 with the arithmetic invariants
     $$\Delta_E=3^3AB^3(ab^3)^p, \; \displaystyle j_E=\frac{3^3c^3(9Aa^p+Bb^p)^3}{AB^3(ab^3)^p}$$ and $$c_4(E)=3^2c(9Aa^p+Bb^p),\; c_6(E)=-3^3(3^3c^6-2^23^2c^3Bb^p+2^3B^2b^{2p}).$$	
	
    The following result is essential for the explicit computations in the proof of Theorem \ref{thm:sample}.  We only need this result for the number fields $K$ stated in Theorem \ref{thm:sample} but in fact the result holds for any  number field satisfying the assumptions.  Below, $K$ is a number field such that there is a unique prime $\lambda$ over $3$. Also, we assume that $(Aa,Bb,c)$ is primitive where $(a,b,c)\in \calO_K^{3}$ is a putative solution of the above curve.

	  \begin{lemma}\label{semist} 
	 	Let $\lambda^e=3\calO_K $ where $\calO_K$ is the integer ring of $K$.
	 	The Frey curve $E$ is semistable away from $\lambda$ . 
	 	Moreover, the conductor $\mathcal{N}_E$ attached to the Frey curve $E$ is given by
	 	$$\mathcal{N}_E=\lambda^\epsilon\prod_{\mathfrak{q}\mid ABab, \mathfrak{q} \nmid 3}\mathfrak{q},$$ 
	 	where 
	 	\begin{enumerate}
	 		\item $\epsilon=\{0,1\}$ if $\lambda \mid Bb$ and $v_\lambda(Bb^p)\ge 3e$.
	 		\item $\epsilon=\{2,3\}$ if $\lambda \nmid AaBb$ and $e=1$.
	 		\end{enumerate}

	 	The Serre conductor $\frakN_E$ is supported on
	 	$\lambda AB$ and belongs to a finite set depending only on the
	 	field $K$.  
	 \end{lemma}
	 
	 \begin{proof}
	 	 Recall that the invariants $c_4(E), c_6(E)$ and $\Delta_E$ of the model $E$ are given by
	 	$$c_4(E)= 9c(9Aa^p+Bb^p), \; c_6(E)=-3^3(3^3c^6-2^23^2c^3Bb^p+2^3B^2b^{2p}), \; \Delta_E=3^3AB^3(ab^3)^p.$$

	 	Suppose $\frakq \neq \lambda$ divides $\Delta_E$, which implies that  $AaBb$  is divisible by $\frakq$.   Without loss of generality, assume that $\frakq\mid Bb$.  We then see that $\frakq$ does not divide $Aa$.  Assume this is not true i.e. $\frakq \mid Aa$, then we have $\frakq \mid c$ since $\frakq \mid Aa^p+Bb^p$ but this contradicts to the primitiveness of $(Aa,Bb,c)$.  By a similar argument we also get that $\frakq$ cannot divide $c$.  Therefore, $c_4(E)= 9c(9Aa^p+Bb^p)$ is not divisible by $\frakq$, i.e. $v_{\frakq}(c_4(E))=0$ and $v_{\frakq}(\calN_E)=1$.  Hence, the given model is minimal and $E$ is semistable at $\frakq$.

	 	Now assume that  $\lambda$ divides $AaBb$. Note that  $\lambda$  can only divide one of $Aa$ or $Bb$ and also $\lambda$ cannot divide $c$. Without loss of generality say $\lambda \mid Bb$. Then we calculate $v_{\lambda}(\mathcal{N}_E)=\epsilon$ via \cite[Tableau~III]{pap}.  Note that the equation is not minimal. After using the change of variables $X=3^2x, Y=3^3y$ we get $v_\lambda(c_4(E))=v_\lambda(c_6(E))=0$ when  $v_\lambda(Bb^p)\geq 2e$, hence $\epsilon=0,1$.  We use \cite[Tableau~II]{pap} to deal with the case $\lambda \nmid AaBb$ and we obtain  $v_\lambda(\mathcal{N}_E)=\epsilon\in\{2,3\}$ when $e=1$ .  

        Note that the assertion on the Serre conductor  $\frakN_E$ can be deduced from Proposition \ref {Conductor of frey} when we take $C=1.$
	 \end{proof}

    The following is our main tool to detect when an elliptic curve has potentially multiplicative reduction. One of the directions is proven using the theory of the Tate curve (see \cite[Proposition V.6.1]{advanced_silverman}). The other direction follows from the fact that if $E$ has potentially good reduction at $\mathfrak{q}\nmid p$, then $24\mid\#\overline{\rho}_{E,p}(I_\mathfrak{q})$ \cite[Section 1]{Kraus90}.
\begin{lemma}\cite[Lemma 3.4]{FS}\label{p divides image of inertia}
    Let $E/K$ be an elliptic curve with $j$-invariant $j_E$. Let $p\geqslant 5$ be a rational prime and $\mathfrak{q}\nmid p$ a prime of $K$. Then $p\mid\#\overline{\rho}_{E,p}(I_\mathfrak{q})$ if and only if $p\nmid v_\mathfrak{q}(j_E)$ and $E$ has potentially multiplicative reduction at $\mathfrak{q}$. 
\end{lemma}
Define $$S_K = \{\mathfrak{L}:\mathfrak{L}\text{ is a prime in }K\text{ and }\mathfrak{L}\mid 3\}.$$ Then Lemma \ref{p divides image of inertia} helps us achieve the following result. This result along with Lemma \ref{p divides image of inertia} is used in Section \ref{level lowering} to get a grasp on the reduction of the level lowered curve at primes in $S_K$.
\begin{lemma}\label{p divides inertia of frey}
    Let $E$ be the Frey curve in \eqref{the frey curve} corresponding to a solution $(a,b,c)\in W_K$. Suppose that $$p>\max_{\mathfrak{L}\in S_K}\{v_\mathfrak{L}(3^2A),v_\mathfrak{L}(B),v_\mathfrak{L}(C),|v_\mathfrak{L}(3) + v_\mathfrak{L}(BA^{-1})|, |3v_\mathfrak{L}(3) + v_\mathfrak{L}(AB^{-1})|\}.$$ Then for every $\mathfrak{L}\in S_K$, we have $p\mid\#\overline{\rho}_{E,p}(I_\mathfrak{L})$. 
\end{lemma}
\begin{proof}
Rewrite the $j$-invariant of $E$ as $$j = 3^3\frac{Cc^3(3^2Aa^p + Bb^p)^3}{AB^3(ab^3)^p} = 3^3\frac{(Aa^p+Bb^p)(3^2Aa^p + Bb^p)^3}{AB^3(ab^3)^p}.$$
We aim to apply Lemma \ref{p divides image of inertia}. Let $\mathfrak{L}$ be a prime in $S_K$. Since $p>v_\mathfrak{L}(C)$, we have that $\mathfrak{L}$ can divide only one of $a$ or $b$. Indeed, if  $\mathfrak{L}$ divides both $a$ and $b$, we obtain $\mathfrak{L}^p\mid Aa^p + Bb^p = Cc^3$. Since $p>v_\mathfrak{L}(C)$, it follows that $\mathfrak{L}\mid c$ which contradicts to the primitivity of $(a,b,c)$. If $\mathfrak{L}\mid a$, then $\mathfrak{L}\nmid b$ and since $p>v_\mathfrak{L}(B)$,
\begin{align*}
v_\mathfrak{L}(j) &= 3v_\mathfrak{L}(3) + v_\mathfrak{L}(B) + 3v_\mathfrak{L}(B) -v_\mathfrak{L}(AB^3) - pv_\mathfrak{L}(a) \\  
&= 3v_\mathfrak{L}(3) + v_\mathfrak{L}(BA^{-1}) - pv_\mathfrak{L}(a)
\end{align*}
and since $p>|3v_\mathfrak{L}(3) + v_\mathfrak{L}(BA^{-1})|$, it follows $v_\mathfrak{L}(j) < 0$ and $p\nmid v_\mathfrak{L}(j)$.

If $\mathfrak{L}\mid b$, then $\mathfrak{L}\nmid a$, and since $p>v_\mathfrak{L}(3^2A)$, 
\begin{align*}
v_\mathfrak{L}(j) &= 3v_\mathfrak{L}(3) + v_\mathfrak{L}(A) + 6v_\mathfrak{L}(3) + 3v_\mathfrak{L}(A) - v_\mathfrak{L}(AB^3) - 3pv_\mathfrak{L}(b) \\ &= 3\big(3v_\mathfrak{L}(3) + v_\mathfrak{L}(AB^{-1}) - pv_\mathfrak{L}(b)\big)
\end{align*}
and since $p>|3v_\mathfrak{L}(3) + v_\mathfrak{L}(AB^{-1})|$, it follows that $v_\mathfrak{L}(j) < 0$ and $p\nmid v_\mathfrak{L}(j)$. In both cases, we get $p\mid\#\overline{\rho}_{E,p}(I_\mathfrak{L})$ from Lemma \ref{p divides image of inertia} .
\end{proof}

\section{Properties of Galois representations}
\subsection{Level lowering}\label{level lowering}
In this subsection, as is standard, we will be relating the Galois representation attached to the Frey curve with another representation of lower level.  That is, we construct a new elliptic curve $E'$ which relates to $E$ via its mod $p$ Galois representation. This new $E'$ will have conductor $\mathfrak{N}_E$ (obtained after level lowering) which does not depend on $a$ and $b$. It is then the curve $E'$ that will lead us to a contradiction. We construct this $E'$ using the modularity conjectures from Section \ref{Conjectures}. The properties of $E'$ are summarized in the following theorem and the rest of this section is dedicated to proving this result. The idea of the proof is heavily inspired by that of \cite[Section 3]{IKO2} and can be seen as a generalization thereof. 
\begin{theorem}\label{the theorem describing E'}
    Let $K$ be a number field satisfying Conjectures \ref{conj1} and \ref{conj2}. Then, there is a constant $V=V(K,A,B,C)$ depending only on $K,A,B$ and $C$ such that the following holds. Let $(a,b,c)\in W_K$ be a solution to \eqref{maineqn} with prime exponent $p>V$. Let $E/K$ be the associated Frey curve defined in \eqref{the frey curve}. Then there is an elliptic curve $E'/K$ such that the following statements hold:
    \begin{enumerate}[noitemsep]
        \item [(i)] $E'$ has good reduction away from $T_K$;
        \item [(ii)] $E'$ has a $K$-rational point of order $3$;
        \item [(iii)] $\overline{\rho}_{E',p}\sim\overline{\rho}_{E,p}$;
        \item [(iv)] for every prime $\mathfrak{L}\in S_K$ we have $v_\mathfrak{L}(j_{E'})<0$. 
    \end{enumerate}
\end{theorem}
Before we dive into the proof of Theorem \ref{the theorem describing E'}, we cover some required results.

The following well-known result about subgroups of $\GL_2(\mathbb F_p)$ will be frequently used.
	
	\begin{theorem}\label{subgroups} Let $E$ be an  elliptic curve over a number field $K$ of degree $d$ and let $G \leq \GL_2(\mathbb F_p)$ be the
		image of the mod $p$ Galois representation of $E$.
		Then the following holds:
		\begin{itemize}
			\item if $p \mid  \#G$ then either $G$ is reducible or $G$ contains $\SL_2(\mathbb F_p)$, hence it is absolutely irreducible. 
			\item if $p \nmid \#G$ and $p > 15 d +1$ then $G$ is contained in a Cartan subgroup or $G$ is contained in the normalizer of Cartan subgroup but not the Cartan subgroup itself.
			
		\end{itemize}
	\end{theorem}

	\begin{proof}
		The main reference of the proof is \cite[Lemma 2]{SD}. The version above including the proof of the second part is from \cite[Propositions 2.3 and 2.6]{localglobal}.
	\end{proof}

The following is Proposition 6.1 from \cite{SS}.

\begin{proposition}\label{irreducibility for Galois NF}
    Let $L$ be a Galois number field, and let $\mathfrak{q}$ be a prime of $L$. There is a constant $B_{L,\mathfrak{q}}$ such that the following is true. Let $p>B_{L,\mathfrak{q}}$ be a rational prime. Let $E/L$ be an elliptic curve that is semistable at all $\mathfrak{p}\mid p$ and having potentially multiplicative reduction at $\mathfrak{q}$. Then $\overline{\rho}_{E,p}$ is irreducible.
\end{proposition}
Applying this to our Frey curve \eqref{the frey curve}, we obtain the first ingredient to prove Theorem \ref{the theorem describing E'}. 
\begin{cor}\label{surjective for p>CK}
    There is a constant $D = D(K,A,B,C)$ depending only on $K,A,B$ and $C$ such that the following holds. If $E$ is the Frey curve \eqref{the frey curve} corresponding to a solution $(a,b,c)\in W_K$ with exponent $p>D$, then the mod $p$ Galois representation $\overline{\rho}_{E,p}$ is surjective. 
\end{cor}
\begin{proof}
    From Lemma \ref{p divides image of inertia} and Lemma \ref{p divides inertia of frey}, for sufficiently large $p$, the curve $E$ has potentially multiplicative reduction at the primes in $S_K$. Additionally, $E$ is semistable outside of $T_K$ by Proposition \ref{Conductor of frey}. Let $L$ be the Galois closure of $K$ and let $G_L=\Gal(\bar{L}/L)$ denote the absolute Galois group of $L$. And let $\mathfrak{q}$ in be a prime in $L$ above any prime in $S_K$, say, $\mathfrak{L}$. Let $B_{L,\mathfrak{q}}$ be the constant from Proposition \ref{irreducibility for Galois NF}. If we let $p$ be large enough such that no prime in $T_K$ lies above $p$, say, $p> c=c(K,A,B,C)$, then by Proposition \ref{Conductor of frey}, $E$ is semistable at the primes above $p$. If we enlarge $p$ further such that $p>B_{L,\mathfrak{q}}$, it follows that $\overline{\rho}_{E,p}(G_L)$ is irreducible. There are only finitely many primes $\mathfrak{q}$ above $\mathfrak{L}$ and $L$ depends only on $K$. So taking $$D = \max_{\mathfrak{q}\mid \mathfrak{L}}\{B_{L,\mathfrak{q}},c\}$$ gives us that $\overline{\rho}_{E,p}\colon G_L\to \GL_2(\mathbb{F}_p)$ is irreducible whenever $p> D$. Since $G_L$ is a subgroup of $G_K$, it follows that $\overline{\rho}_{E,p}\colon G_K\to \GL_2(\mathbb{F}_p)$ is irreducible whenever, $p>D$. 
    If necessary, enlarge $D$ so that, by Lemma \ref{p divides inertia of frey}, we have  $p\mid\#\overline{\rho}(G_K)$. 
    By Theorem \ref{subgroups} and the fact that $\overline{\rho}_{E,p}$ is irreducible, it follows that $\overline{\rho}_{E,p}(G_K)$ contains $\text{SL}_2(\mathbb{F}_p)$. It now follows that $\overline{\rho}_{E,p}$ is surjective when $\chi_p = \det\overline{\rho}_{E,p}$ is surjective. This can be ensured by taking $D$ large enough so that $\zeta_p\not\in K$.
\end{proof}

The following results are very technical results relating eigenforms to elliptic curves. Not everything is known about this connection (over general number fields) so this is where the conjectures come in.

\begin{proposition}\label{lifting of forms}\cite[Proposition 2.1]{SS}
    Let $\mathfrak{N}$ be an ideal of $\mathcal{O}_K$. There is an integer $B(\mathfrak{N})$ depending only on $\mathfrak{N}$, such that, for any prime $p>B(\mathfrak{N})$, every weight two, mod $p$ eigenform of level $\mathfrak{N}$ lifts to a complex one.
\end{proposition}

Proposition \ref{lifting of forms} allows us to prove the following result which will play a major role in the proof of Theorem \ref{the theorem describing E'}.

\begin{lemma}\label{curve to form}
    Suppose $K$ satisfies Conjectures \ref{conj1} and \ref{conj2}. There is a constant $V=V(K,A,B,C)$ depending only on $K,A,B$ and $C$ such that whenever $p>V$, the following holds. There is a non-trivial, new, weight two, complex eigenform $\mathfrak{f}$ which has an associated elliptic curve $E'=E_\mathfrak{f}$ of conductor $\mathfrak{N}'$ dividing $\mathfrak{N}_E$ and $\overline{\rho}_{E',p}\sim\overline{\rho}_{E,p}$.
\end{lemma}
\begin{proof}
    If we take $p$ large enough, then by Corollary \ref{surjective for p>CK}, $\overline{\rho}_{E,p}$ is surjective and hence absolutely irreducible. According to Proposition \ref{Conductor of frey}, $\overline{\rho}_{E,p}$ satisfies every condition of Conjecture \ref{conj1}. Applying this conjecture gives a weight two mod $p$ eigenform $\theta$ over $K$ of level $\mathfrak{N}_E$, such that for all primes $\mathfrak{q}$ coprime to $p\mathfrak{N}_E$, we have \begin{equation}\text{Tr}(\overline{\rho}_{E,p}(\text{Frob}_\mathfrak{q})) = \theta(T_\mathfrak{q}).\label{trace of rep}\end{equation} From Proposition \ref{Conductor of frey}, we know that there is only a finite amount of values of $\mathfrak{N}_E$ possible. Therefore, by Proposition \ref{lifting of forms} we know that we can take $p$ large enough so that for any of the possible values of $\mathfrak{N}_E$ there is weight two complex eigenform $\mathfrak{f}$ with level $\mathfrak{N}_E$ that is a lift of $\theta$. There are only finitely many such eigenforms $\mathfrak{f}$ and they depend only on $K,A,B$ and $C$, so any constant depending on $\mathfrak{f}$ also only depends on these invariants. \\ 

    Next, we aim to apply Conjecture \ref{conj2}. We have that $\overline{\rho}_{E,p}$ is irreducible, from this it follows that $\mathfrak{f}$ is non-trivial. If $\mathfrak{f}$ is not new then we can replace $\mathfrak{f}$ with an equivalent new eigenform of level $\mathfrak{N}'$ dividing $\mathfrak{N}_E$. Therefore, we may assume that $\mathfrak{f}$ is new of level $\mathfrak{N}'$ dividing $\mathfrak{N}_E$. By applying Conjecture \ref{conj2}, we obtain that $\mathfrak{f}$ is either associated to an elliptic curve $E_\mathfrak{f}/K$ of conductor $\mathfrak{N}'$, or has an associated fake elliptic curve $A_\mathfrak{f}/K$ of conductor $\mathfrak{N}'^2$. By Lemma \ref{bullet} below, we can assume that we are in the former case by taking $p$ sufficiently large. Denote $E'=E_\mathfrak{f}$, then we have $\overline{\rho}_{E,p}\sim\overline{\rho}_{E',p}$. To see this, from Conjecture \ref{conj2}, we have that for all primes $\mathfrak{q}\nmid\mathfrak{N}'$, $$\mathfrak{f}(T_\mathfrak{q}) = 1 + \text{Norm}(\mathfrak{q}) - \# E'(\mathcal{O}_K/\mathfrak{q}).$$
    Reducing this mod $p$, we find $$\text{Tr}(\overline{\rho}_{E',p}(\text{Frob}_\mathfrak{q})) = \theta(T_\mathfrak{q}) \stackrel{\eqref{trace of rep}}{=} \text{Tr}(\overline{\rho}_{E,p}(\text{Frob}_\mathfrak{q})).$$
    Since the set of elements of the form $\text{Frob}_\mathfrak{q}$ with $\mathfrak{q}\nmid p\mathfrak{N}_E$ are a dense subset of $G_K$, it follows that $\text{Tr}(\overline{\rho}_{E,p}) = \text{Tr}(\overline{\rho}_{E',p})$ and since also $\det\overline{\rho}_{E',p} = \chi_p = \det\overline{\rho}_{E,p}$, it follows that $\overline{\rho}_{E',p}\sim \overline{\rho}_{E,p}$.
\end{proof}

\begin{lemma}\cite[Lemma 7.3]{SS}\label{bullet}
    If $p>24$, then $\mathfrak{f}$ has an associated elliptic curve $E_\mathfrak{f}$. 
\end{lemma}
The following lemma will allow us to take $p$ large enough so that $E'$ has a $K$-rational point of order 3. 
\begin{lemma}\label{that good}
    If $E'$ as in Lemma \ref{curve to form} does not have a non-trivial $K$-rational point of order $3$ and is not isogenous to an elliptic curve with a non-trivial $K$-rational point of order $3$, then $p<C_{E'}$ where $C_{E'}$ is a constant depending only on $E'$. 
\end{lemma}
\begin{proof}
    By \cite[Theorem 2]{Katz}, there are infinitely many primes $\mathfrak{P}$ such that $\#E'(\mathcal{O}_K/\mathfrak{P})\not\equiv 0 \text{ mod } 3$. Fix such a prime $\mathfrak{P}\not\in T_K$.  The conductor of $E'$ is supported on the primes in $T_K$ by Proposition \ref{Conductor of frey}. Therefore $E'$ is semistable at $\mathfrak{P}$. Suppose that $E'$ has good reduction at $\mathfrak{P}$. Then since $\mathfrak{P}\nmid\mathfrak{N}'$, we have $\text{Tr}(\overline{\rho}_{E,p}(\text{Frob}_\mathfrak{P})) \equiv \text{Tr}(\overline{\rho}_{E',p}(\text{Frob}_\mathfrak{P}))\text{ mod } p$, or equivalently, $\#E(\mathcal{O}_K/\mathfrak{P}) \equiv \#E'(\mathcal{O}_K/\mathfrak{P})\text{ mod } p$. Since $3\mid\# E(\mathcal{O}_K/\mathfrak{P})$, the difference is nonzero. Since the difference is divisible by $p$, it belongs to a finite set. This gives a bound on $p$. If $E$ has multiplicative reduction at $\mathfrak{P}$, then we have $$\pm(\text{Norm}(\mathfrak{q} + 1)) \equiv a_\mathfrak{P}(E') \mod p.$$ By comparing the traces of Frobenius, we get that the difference belongs to a bounded set which gives a bound on $p$. 
\end{proof}
With all the ingredients gathered, we can now prove Theorem \ref{the theorem describing E'}.

\begin{proof}[\textbf{Proof of Theorem \ref{the theorem describing E'}}]
    By assuring that $p$ is large enough, we can invoke Lemma \ref{curve to form} to get an elliptic curve $E' = E_\mathfrak{f}$. It remains to show that $E'$ satisfies \textit{(i)}-\textit{(iv)}. The elliptic curve $E'$ has conductor $\mathfrak{N}'$ dividing $\mathfrak{N}_E$, which is supported on the primes in $T_K$. Thus, $E'$ has good reduction outside of $T_K$ giving \textit{(i)}. Suppose that $E'$ does not have a $K$-rational point of order $3$ and is not 3-isogenous to an elliptic curve with a $K$-rational point of order $3$. Then by Lemma \ref{that good}, $p<C_{E'}$. Therefore, by taking $p$ large enough and replacing $V$ by $\max\{V,C_{E'}\}$ it follows that either $E'$ has a $K$-rational point of order $3$ or $E'$ is 3-isogenous to an elliptic curve $E''$ with a $K$-rational point of order $3$. In the latter case, we have that for every prime $\ell\neq3$, the isogeny induces an isomorphism $E'[\ell]\cong E''[\ell]$ so $\overline{\rho}_{E',p}\sim \overline{\rho}_{E'',p}$ and since $\overline{\rho}_{E,p}\sim\overline{\rho}_{E',p}$ we obtain \textit{(ii)} and \textit{(iii)} after possibly replacing $E'$ by $E''$. To show that \textit{(iv)} holds, let $\mathfrak{P}\in S_K$. As we have $\overline{\rho}_{E,p}\sim\overline{\rho}_{E',p}$, we also have  $\#\overline{\rho}_{E,p}(I_\mathfrak{P})=\#\overline{\rho}_{E',p}(I_\mathfrak{P})$ and  it follows that $p\mid\#\overline{\rho}_{E',p}(I_\mathfrak{P})$ from Lemma \ref{p divides inertia of frey} and  we obtain $v_\mathfrak{P}(j_{E'}) < 0$ by Lemma \ref{p divides image of inertia}. This gives \textit{(iv)} and concludes the proof.
\end{proof}

    \subsection {Irreducibility and absolute irreducibility of Galois Representations}
    Throughout this section $K=\Q(\sqrt{-d})$, where $d \in \{7,19,43,67\}$, $(a,b,c) \in \mathcal O_K^3$ is a non-trivial, putative solution of the equation $Aa^p+Bb^p=c^3$ such that $(Aa,Bb,c)$ is primitive. 

	The idea of this section is to prove that when $p$ is bigger than an explicit constant, then the mod $p$ Galois representation $\overline{\rho}_{E,p}:G_\Q \rightarrow \Aut(E[p]) \cong \GL_2(\F_p)$ attached to $E$ is absolutely irreducible.

	With the above assumptions, the following theorem is a simpler version of  \cite[Proposition 3.9]{IKO2} which is based on \cite[Lemma 6.3]{FSANT}. Therefore, we skip its proof and refer the reader to \cite{IKO2}.
    
	\begin{proposition}\label{irrCase2}
		Let $E/K$ be the Frey curve attached to a putative solution to Equation \ref{maineqn} where $\lambda \mid  Bb$  and let $p > 20$ be a prime. Then $\overline{\rho}_{E,p}$  is irreducible. 
	\end{proposition}

    In order to apply Conjecture \ref{conj1}, we need absolute irreducibility of the Galois representation. The following result ensures that we have this.
    
	\begin{cor}\label{absirr} \begin{enumerate}
			\item Let $p$ be odd. If $\lambda \mid  ABab$ and $\overline{\rho}_{E,p}$ is irreducible, then $\overline{\rho}_{E,p}$ is absolutely irreducible. 
			
			\item Let $p$ be odd. If $K$ is totally real, then $\overline{\rho}_{E,p}$ is irreducible if and only if it is absolutely irreducible. 
		\end{enumerate}
		\end{cor}

	\begin{proof} 
		\begin{enumerate}
			\item 
			Let  $(a,b,c)\in \mathcal O_K^3$ be a putative solution  of $Ax^p+By^p=z^3$ such that  $(Aa,Bb,c)$ is primitive. Then the Frey curve $E$ attached to such a solution is semistable when $\lambda \mid  ABab$ where $\lambda$ is the prime of $\mathcal O_K$ lying over $3$. These were discussed in Lemma \ref{semist}. By Proposition \ref{irrCase2} we know that $\overline{\rho}_{E,p}$ is irreducible when $p$ is big enough. In Proposition \ref{irrCase2} we make this bound explicit. By Lemma \ref{p divides image of inertia} the inertia group $I_{\lambda}$ contains an element of order $p$. By Theorem \ref{subgroups}, the image $\overline{\rho}_{E,p}$  contains $\SL_2(\mathbb F_p)$ and is therefore an absolutely irreducible group of $\GL_2(\mathbb F_p)$.

			\item When $K$ is totally real, the absolute Galois group $G_K$ contains a complex conjugation. The image of this complex conjugation under $\overline{\rho}_{E,p}$ is similar to $\begin{pmatrix}
			1 & 0 \\
			0 & -1
			\end{pmatrix}$ which implies that if $\overline{\rho}_{E,p}$ is irreducible, then it is absolutely irreducible.  
		\end{enumerate}
	\end{proof}

\section{Asymptotic Results}\label{T-unit section}
\subsection{Proof of Theorem \ref{main theorem}}
In this subsection we prove Theorem \ref{main theorem} by using the level lowered curve $E'$ to arrive at a contradiction. This is done using $S$-units. First we establish some terminology. Let $S$ be a finite set of primes. We say that an ideal $I\subset\mathcal{O}_K$ is an $S$-ideal if it is only divisible by ideals in $S$. We denote by $\mathcal{O}_S$ the ring of $S$-integers and by $\mathcal{O}_S^\times$ the set of $S$-units. The following reasoning is essentially the same as in \cite[Section 4.4]{moc22} and \cite[Section 2.5]{Kumar}. \\  

\begin{proof}[\textbf{Proof of Theorem \ref{main theorem}}]Let $(a,b,c)\in W_K$ be a solution to Equation \ref{maineqn}. Then by Theorem \ref{the theorem describing E'} there is an elliptic curve $E'/K$ such that $E'$ has a $K$-rational point of order $3$. Then $E'$ has a model of the form 
$$E'\colon y^2 + exy + dy = x^3$$ for some $d,e\in K$ with $j$-invariant $$j_{E'} = \frac{e^3(e^3-24d)^3}{d^3(e^3-27d)}.$$ By Theorem \ref{the theorem describing E'}, $E'$ has good reduction away from $T_K$ which implies that $j_{E'}\in\mathcal{O}_{T_K}$. Set $\lambda=\tfrac{e^3}{d}$ and $\mu= \lambda - 27$. Then 
\begin{equation}
  j_{E'} = \frac{\lambda(\lambda-24)^3}{\lambda-27} = \frac{(\mu+27)(\mu+3)^3}{\mu}  = \mu^{3}(1+27\mu^{-1})(1+3\mu^{-1})^3.\label{jE'}
\end{equation}
Rearranging the equations above and using the fact that $j_{E'}\in\mathcal{O}_{T_K}$, we see that $\lambda$ and $\mu$ satisfy monic polynomials with coefficients in $\mathcal{O}_{T_K}$. Since $\mathcal{O}_{T_K}$ is integrally closed, it follows that $\lambda$ and $\mu$ are elements in $\mathcal{O}_{T_K}$. Using the right hand side of \eqref{jE'} we then see that $\mu^{-1}$ also satisfies a monic polynomial with coefficients in $\mathcal{O}_{T_K}$. It then follows that $\mu^{-1}\in \mathcal{O}_{T_K}$ and hence $\mu\in\mathcal{O}_{T_K}^\times$. By Lemma 17.(ii) of \cite{moc22}, the prinicipal ideal $(\lambda)$ is equal to $I^3J$ for some fractional ideal $I$ and $T_K$-ideal $J$. Since $J$ is a $T_K$-ideal, we have $[I]^3 = 1$ in $\text{Cl}_{T_K}(K)$. By assumption, $\text{Cl}_{T_K}(K)$ has no $3$-torsion so we find that $I = \gamma\tilde{I}$ for some $T_K$-ideal $\tilde{I}$ and $\gamma\in\mathcal{O}_K$. So $$(\lambda) = \gamma^3\tilde{I}J \quad\Leftrightarrow\quad \Big(\frac{\lambda}{\gamma^3}\Big) = \tilde{I}J.$$
The right hand side of the latter is a $T_K$-ideal so it follows that $u\coloneqq\lambda/\gamma^3\in\mathcal{O}_{T_K}^\times$. Recall that we have $\mu+27 = \lambda$, dividing this equation by $u$ gives $$\alpha + \beta = \gamma^3$$ where $\alpha = \mu/u$ and $\beta = 27/\mu$ which are both elements of $\mathcal{O}_{T_K}^\times$. By assumption, there is some $\mathfrak{L}\in S_K$ such that $$|v_\mathfrak{L}(\tfrac{\mu}{27})| = |v_\mathfrak{L}(\alpha\beta^{-1})| \leqslant 3v_\mathfrak{L}(3).$$ This is equivalent to saying that $0\leqslant v_\mathfrak{L}(\mu)\leqslant 6v_\mathfrak{L}(3)$. We show that these bounds on $v_\mathfrak{L}(\mu)$ imply that $v_\mathfrak{L}(j_{E'})\geqslant 0$. From the expression in \eqref{jE'} in $\mu$, we find that
\begin{equation}
    v_\mathfrak{L}(j_{E'}) = v_\mathfrak{L}(\mu+27) + 3v_\mathfrak{L}(\mu+3) - v_\mathfrak{L}(\mu)\label{vLj}.
\end{equation}

We distinguish three cases. \\ 

First suppose that $0\leqslant v_\mathfrak{L}(\mu)\leqslant v_\mathfrak{L}(3)$. Then $v_\mathfrak{L}(\mu+27) = v_\mathfrak{L}(\mu)$ and $v_\mathfrak{L}(\mu+3)\geqslant v_\mathfrak{L}(\mu)$. Then \eqref{vLj} implies that $v_\mathfrak{L}(j_{E'})\geqslant 0$. \\ 

If $v_\mathfrak{L}(3) < v_\mathfrak{L}(\mu) \leqslant 3v_\mathfrak{L}(3)$, then $v_\mathfrak{L}(\mu+27) \geqslant v_\mathfrak{L}(\mu)>v_\mathfrak{L}(3)$ and $v_\mathfrak{L}(\mu+3) = v_\mathfrak{L}(3)$. Then \eqref{vLj} implies that $v_\mathfrak{L}(j_{E'})> 0$. \\

Finally, if $3v_\mathfrak{L}(3) < v_\mathfrak{L}(\mu) \leqslant 6v_\mathfrak{L}(3)$, then $v_\mathfrak{L}(\mu+27) = 3v_\mathfrak{L}(3)$ and $v_\mathfrak{L}(\mu+3) = v_\mathfrak{L}(3)$. Then $v_\mathfrak{L}(j_{E'}) = 6v_\mathfrak{L}(3) - v_\mathfrak{L}(\mu) \geqslant 0$. In all cases, this contradicts with Theorem \ref{the theorem describing E'}.(iv). \end{proof} 

\subsection{Proof of Theorem \ref{theorem: special case}}
As in \cite{moc22}, under some stricter conditions on $K$, $A$, $B$ and $C$, we may replace the $S$-unit equation in the statement of Theorem \ref{main theorem} by a simpler one. The proof of this theorem is again similar to that of \cite[Theorem 11]{moc22} and \cite[Proposition 2.7]{Kumar}. 

\begin{proof} [\textbf{Proof of Theorem \ref{theorem: special case}}]
    With notation as in Theorem \ref{theorem: special case}, we have $T_K = \{\mathfrak{q}\}$. Note that since $3\nmid h_K$, by Theorem \ref{main theorem}, it suffices to show that for every solution $(\alpha,\beta,\gamma)\in \mathcal{O}_{S_K}^\times\times\mathcal{O}_{S_K}^\times\times\mathcal{O}_{S_K}$ to $\alpha+\beta = \gamma^3$ satisfies $|v_\mathfrak{q}(\alpha\beta^{-1})|\leqslant 3v_\mathfrak{q}(3)$. Let $(\alpha,\beta,\gamma)$ be such a solution. Scale this solution by cubic powers (this is where the assumption that $\mathfrak{q}$ is principal comes in) and swap $\alpha$ and $\beta$ if necessary so that $v_\mathfrak{q}(\beta) = 0,1$ or $2$ and $0\leqslant v_\mathfrak{q}(\beta)\leqslant v_\mathfrak{q}(\alpha)$. By scaling with $-1$, if necessary, we may also assume $\beta$ to be positive. We consider several cases.
    
    \textbf{Case 1:} Suppose that $v_\mathfrak{q}(\beta) = 1$ or $2$. If $v_\mathfrak{q}(\alpha)\neq v_\mathfrak{q}(\beta)$, then $v_\mathfrak{q}(\alpha)>v_\mathfrak{q}(\beta)$ and $$v_\mathfrak{q}(\gamma^3) = v_\mathfrak{q}(\alpha+\beta) = v_\mathfrak{q}(\beta)$$
    which is a contradiction since $v_\mathfrak{q}(\beta)$ is not a multiple of $3$. Therefore, $v_\mathfrak{q}(\alpha) = v_\mathfrak{q}(\beta)$. It follows that $|v_\mathfrak{q}(\alpha\beta^{-1})| = 0\leqslant 3v_\mathfrak{q}(3)$.

    \textbf{Case 2:} Suppose that $v_\mathfrak{q}(\beta) = 0$ and that $\beta$ is not a cube in $K$. Suppose for a contradiction that $v_\mathfrak{q}(\alpha) > 3v_\mathfrak{q}(3)$. Let $L = K(\zeta_3)$ and consider the extension $L(\sqrt[3]{\beta})/L$. Let $\mathfrak{Q}$ be a prime above $\mathfrak{q}$ in $L$. We show that this extension is unramified above $\mathfrak{Q}$ by showing that $\beta$ is a cube in the $\mathfrak{q}$-adic completion $K_\mathfrak{q}$ of $K$. To do this, we use a general version of Hensel's Lemma \cite[Lemma 9.16]{MITHensel}. Let $f = X^3 - \beta \in K[X]$. Since $v_\mathfrak{q}(\alpha) > 3v_\mathfrak{q}(3)$ and $v_\mathfrak{q}(\beta) = 0$ it follows that $\gamma^3 = \alpha + \beta \equiv \beta \mod 3^3$. Thus, $v_\mathfrak{q}(f(\gamma)) \geqslant 3v_\mathfrak{q}(3)$. We also have $3v_\mathfrak{q}(\gamma) = \min\{v_\mathfrak{q}(\alpha),v_\mathfrak{q}(\beta)\} = 0$. From this it follows that $v_\mathfrak{q}(f'(\gamma)) = v_\mathfrak{q}(3) + 2v_\mathfrak{q}(\gamma) = v_\mathfrak{q}(3)$. We have $$v_\mathfrak{q}(f(\gamma)) \geqslant 3v_\mathfrak{q}(3) > 2v_\mathfrak{q}(3) = 2v_\mathfrak{q}(f'(\gamma)).$$ From Hensel's Lemma \cite[Lemma 9.16]{MITHensel} it follows that $f$ has a root in $K_\mathfrak{q}$. In other words, $\beta$ is a cube in $K_\mathfrak{q}$. In particular, $\beta$ is also a cube in the $\mathfrak{Q}$-adic completion $L_\mathfrak{Q}\supset K_\mathfrak{q}$ of $L$. From this it follows that $L_\mathfrak{Q}(\sqrt[3]{\beta})/L_\mathfrak{Q}$ is the trivial extension and hence unramified. Therefore, since $\mathfrak{Q}$ was taken to be an arbitrary prime above $\mathfrak{q}$, it follows that $L(\sqrt[3]{\beta})/L$ is unramified above every prime in $L$ dividing $\mathfrak{q}$. The other primes where $L(\sqrt[3]{\beta})/L$ may ramify are the primes dividing $\beta$. However, since $\beta\in\mathcal{O}_{S_K}^\times$, it follows that $\beta$ is supported on the primes above $\mathfrak{q}$. Hence, $L(\sqrt[3]{\beta})/L$ is unramified above all places (also the infinite ones since $L$ has no real embeddings). Further, the extension $L(\sqrt[3]{\beta})/L$ has degree $3$, is Galois, and has a cyclic Galois group. It follows that $3\mid h_L$ which contradicts with our assumption. Therefore, $v_\mathfrak{q}(\alpha) \leqslant 3v_\mathfrak{q}(3)$ and $|v(\alpha\beta^{-1})| = v_\mathfrak{q}(\alpha) \leqslant 3v_\mathfrak{q}(3)$. 

    \textbf{Case 3:} Suppose that $v_\mathfrak{q}(\beta) = 0$ and $\beta$ is a cube. By dividing the equation $\alpha+\beta = \gamma^3$ by $\beta$, we may assume that $\beta = 1$. We get an equation of the form \eqref{alphaplusoneequalsgammacubed} and by assumption we get that $v_\mathfrak{q}(\alpha)\leqslant 3v_\mathfrak{q}(3)$ and hence $|v_\mathfrak{q}(\alpha\beta^{-1})| = v_\mathfrak{q}(\alpha) \leqslant 3v_\mathfrak{q}(3)$.
    In all possible cases we find that $|v_\mathfrak{q}(\alpha\beta^{-1})|\leqslant 3v_\mathfrak{q}(3)$. The result then follows from Theorem \ref{main theorem}. 
\end{proof}
\subsection{Quadratic imaginary number fields and proof of Corollary \ref{corollary: quadratic imaginary fields asymptotic solutions}}
The $S_K$-integer equation occurring in Theorem \ref{theorem: special case} is solvable for quadratic imaginary fields which only have one prime above $3$. This allows us to say something about asymptotic solutions to \eqref{maineqn} for these types of number fields. The following result solves this $S_K$-integer equation. 
\begin{proposition}\label{proposition: S-unit equation for quadratic imaginary}
    Let $d\geqslant2$ be a positive square-free integer such that $d\equiv 1\textup{ mod }3$. Let $K = \mathbb{Q}(\sqrt{-d})$ and let $S$ be the set of primes in $K$ above $3$. The only solutions to the $S$-unit equation $\alpha+1=\gamma^3$ with $(\alpha,\gamma)\in\mathcal{O}_S^\times\times\mathcal{O}_S$ are $(-1,0)$ and $(-9,-2)$. 
\end{proposition}
\begin{proof}
    Let $(\alpha,\gamma)\in\mathcal{O}_S^\times\times\mathcal{O}_S$ be a solution to $\alpha+1=\gamma^3$. The condition that $d\equiv 1\mod 3$ is precisely the condition that $3$ is inert in $K$. Therefore $S = \{(3)\}$, $\mathcal{O}_S = \mathcal{O}_K[\frac{1}{3}]$ and $\mathcal{O}_S^\times = \{\pm1\}\times\langle3\rangle$. Hence, $\alpha = \pm3^n$ for some $n\in\mathbb{Z}$. First suppose that $n\leqslant 0$. Then $\pm3^n = \pm\tfrac{1}{3^k}$ where $k = |n|$. It follows that \begin{equation}\gamma^3 = \frac{\pm1 + 3^k}{3^k}.\label{gamma^3}\end{equation} Therefore, $3v_3(\gamma) = -k$ which shows that $k$ is divisible by $3$ and that $\gamma=\tfrac{c}{3^{k/3}}$ for some $c\in\mathcal{O}_K$. Then, \eqref{gamma^3} shows that $c^3 = \pm1 + 3^k$. We have
    $$\pm1 = c^3-3^k = \big(c-3^{\tfrac{k}{3}}\big)\big(c^2 + 3^{\tfrac{k}{3}}c + 3^{\tfrac{2k}{3}}\big).$$
    This is an equality in $\mathcal{O}_K$ and hence it follows that $c-3^{\tfrac{k}{3}}\in\mathcal{O}_K^\times = \{\pm1\}$. We obtain $c = \pm1 + 3^{\tfrac{k}{3}}$ and $$c^3 = \pm1 + 3^{\tfrac{k}{3}+1} \pm 3^{\tfrac{2k}{3}+1} +3^k\quad\Longleftrightarrow\quad c^3 - 3^k=\pm1 + 3^{\tfrac{k}{3}+1} \pm 3^{\tfrac{2k}{3}+1}.$$ The left hand side of the last equation has absolute value equal to $1$. We distinguish between two cases. If $c - 3^{\tfrac{k}{3}} = 1$, then the above gives that $$|1 + 3^{\tfrac{k}{3}+1} + 3^{\tfrac{2k}{3}+1}| = 1 + 3^{\tfrac{k}{3}+1} + 3^{\tfrac{2k}{3}+1} = 1.$$ This is clearly only possible if $k=0$. Suppose that instead $c - 3^{\tfrac{k}{3}} = -1$. The sequence $$a_n = -1 + 3^{n+1} - 3^{2n+1}$$ is strictly decreasing in $n$. Further, $a_1 = -19$ so if $|a_n| = 1$ we must have $n = 0$. From this we see that the equality $$|-1+3^{\tfrac{k}{3}+1} - 3^{\tfrac{2k}{3} + 1}| = 1$$ implies $k = 0$. In either case, we find that $k=0$. We conclude that $n\geqslant0$. Next, suppose for a contradiction that $n>2$. Since $\gamma^3 = \pm3^n + 1$, we get $v_3(\gamma)=0$ and hence $\gamma\in\mathcal{O}_K$. Write \begin{equation}\pm3^n = (\gamma-1)(\gamma-\zeta_3)(\gamma-\zeta_3^2)\label{equation: factorization of 3^n}\end{equation} in $L \coloneqq K(\zeta_3)$. Define
    \begin{align*}
        x = \gamma-1,\quad y=\gamma-\zeta_3 \quad\text{and}\quad z=\gamma-\zeta_3^2.
    \end{align*}
    Then $x-y = \zeta_3 - 1$ and $y-z = \zeta_3(\zeta_3-1)$. Let $\sigma\in\Gal(L/K)$ be the generator of $\Gal(L/K)$ i.e. $\sigma$ is the element of $\Gal(L/K)$ such that $\sigma(\zeta_3) = \zeta_3^2$. Let $\mathfrak{p} = (\zeta_3-1)$ be the prime above $3$ in $L$. Then $\sigma(\mathfrak{p}) = \mathfrak{p}$ and since $\sigma(z) = y$ we have $v_\mathfrak{p}(z) = v_\mathfrak{p}(y)\eqqcolon r$. Hence, we get  $$1 = v_\mathfrak{p}(\zeta_3(\zeta_3-1)) = v_\mathfrak{p}(y-z)\geqslant r.$$ We claim that $r$ is equal to $1$. Suppose for a contradiction that $r\leqslant 0$. Since we have $v_\mathfrak{p}(3) = 2$, then $$v_\mathfrak{p}(x) \geqslant v_\mathfrak{p}(xyz) = v_\mathfrak{p}(3^n) > 4.$$
    Using this, we obtain $$1=v_\mathfrak{p}(\zeta_3-1) = v_{\mathfrak{p}}(x-y) =\min\{v_\mathfrak{p}(x),v_\mathfrak{p}(y)\} =r\leqslant 0,$$ a contradiction. We conclude that $r=1$. It follows that $2n = v_\mathfrak{p}(xyz) = v_\mathfrak{p}(x) + 2$ and hence $v_\mathfrak{p}(x) > 2$. From the factorization \eqref{equation: factorization of 3^n} and the fact that $n\geqslant0$, we see that $x$ and $y$ are supported on $\mathfrak{p}$. Using these facts, define 
    $$u = \frac{x}{\zeta_3-1}\in\mathcal{O}_L\quad\text{and}\quad v=\frac{-y}{\zeta_3-1}\in\mathcal{O}_L^\times.$$ Note that $u$ is indeed integral since $x$ and $\zeta_3-1$ are only supported on $\mathfrak{p}$ with $v_\mathfrak{p}(x)>2$ and $v_\mathfrak{p}(\zeta_3-1)$. Similarly, $y$ is a unit since $v_\mathfrak{p}(y) = 1$. We have $\mathfrak{p}^2 = (3)\mid u$ and since $u+v = 1$ we get $v\equiv1\mod 3$. Let $\tau$ be the generator of $\Gal(L/\mathbb{Q}(\zeta_3))$. Then, since $\tau(\mathfrak{p}) = \mathfrak{p}$, it follows that $(3)\mid \tau(u)$ and hence also $\tau(v) \equiv 1\mod3$. Therefore, $N_{L/\mathbb{Q}(\zeta_3)}(v)=\tau(v)v\equiv1\mod3$. Since $v\in\mathcal{O}_L^\times$, we also have $N_{L/\mathbb{Q}(\zeta_3)}(v) \in\mathcal{O}_{\mathbb{Q}(\zeta_3)}^\times = \langle-\zeta_3\rangle$. Combining these two facts, it follows that $N_{L/\mathbb{Q}(\zeta_3)}(v) = 1$. Let $F\coloneqq\mathbb{Q}(\sqrt{3d})$ denote the unique totally real subfield of $L$. Suppose that $v\in\mathcal{O}_L^\times\minus \mathcal{O}_F^\times = \zeta_3\mathcal{O}_F^\times$. Then $N_{L/\mathbb{Q}(\zeta_3)}(v)$ is a multiple of $\zeta_3$ which contradicts to $N_{L/\mathbb{Q}(\zeta_3)}(v) = 1$. Therefore, $v\in\mathcal{O}_F^\times$ and hence $u = v-1\in\mathcal{O}_F$. Since $x\in\mathcal{O}_K$ and $u$ is a quotient of $x$ and $\zeta_3-1$, one readily verifies that this is a contradiction. We conclude that $n\in\{0,1,2\}$. Hence, we obtain $$\pm3^n +1\in\{0,2,4,-2,10,-8\}.$$ The only cubes in this set are $0$ and $-8$. This concludes the proof.
\end{proof}
Now, we can prove Corollary \ref{corollary: quadratic imaginary fields asymptotic solutions} as a consequence of the above proposition.
\begin{proof} [\textbf{Proof of Corollary \ref{corollary: quadratic imaginary fields asymptotic solutions}}]
    Let $S = \{(3)\}$ consist of the only prime above $3$ in $K$, from Proposition \ref{proposition: S-unit equation for quadratic imaginary} it follows that the solutions to $\alpha+1 = \gamma^3$ in $\mathcal{O}_S^\times\times\mathcal{O}_S$ are $(-1,0)$ and $(-9,-2)$. The result then follows from Theorem \ref{theorem: special case}.
\end{proof}

	\section{Effective Results: Proof of Theorem \ref{thm:sample}}\label{effres}

     In this section, we will prove Theorem \ref{thm:sample}. Recall that Proposition \ref{lifting of forms}  states that there is a constant $B(\mathfrak{N})$ such that, for any $p>B(\mathfrak{N})$, all weight two, mod $p$ eigenforms lift to complex ones. Therefore, lifting mod $p$ eigenforms to complex ones is guaranteed when $p>B(\mathfrak{N})$.
     
     The main point of the proof is to make this bound explicit for the fields considered in  Theorem \ref{thm:sample}.  The procedure  and the theoretical background behind this lifting argument are explained at the beginning of Section 5 of \cite{IKO2} and we refer the reader there to avoid repetition.
 
     Let $K=\mathbb{Q}(\sqrt{-d})$  where $d \in  \{7,19,43,67\}$ and let $\lambda$ denote the prime ideal of $K$ lying above $3$ as introduced in the statement of Theorem \ref{thm:sample}. Suppose that $(a,b,c)\in\mathcal{O}_K^3$ is a non-trivial solution to the equation $a^p+db^p=c^3$ where $(a,db,c)$ is primitive and $\lambda\mid b$. Let $E$ be the Frey curve attached to such a solution, i.e. the curve given in Equation \ref{the frey curve} and $\overline{\rho}_{E,p}$ be the residual Galois representation induced by the action of $G_K$ on $E[p]$. 
     
     Now, we want to apply Conjecture~\ref{conj1} to $\overline{\rho}_{E,p}$ . Note that, in order to do this we need $\overline{\rho}_{E,p}$  to be absolutely irreducible.  Recall that the prime ideal $\lambda$ of $K$ lying above $3$ divides $b$. In this case, the associated Frey curve is semistable by Lemma \ref{semist}. By Proposition \ref{irrCase2}, $\overline{\rho}_{E,p}$ is irreducible when $p>20$. By Corollary \ref{absirr} Part (1), $\overline{\rho}_{E,p}$ is absolutely irreducible if it is irreducible. Therefore, Conjecture ~\ref{conj1} is applicable to $\overline{\rho}_{E,p}$.

     Using Conjecture ~\ref{conj1},  we deduce that there exists a weight two, mod $p$
eigenform $\theta$ over $K$ of level $\frakN_E$ such that for all primes $\frakq$ coprime to $p\frakN_E$, we have
\[
\text{Tr}(\overline{\rho}_{E,p}(\Frob_{\frakq}))=\theta(T_{\frakq}),
\]
where $T_{\frakq}$ denotes the Hecke operator at $\frakq$.

Recall that $\frakN_E$ denotes the Serre conductor of the residual representation $\overline{\rho}_{E,p}$. By Lemma \ref{semist}, $\frakN_E$ is a power of $\lambda$ times $\frakD$ where $\frakD$ is the unique prime of $K$ lying over $d$. Since $\lambda \mid  b$ and $p$ is big enough, $\frakN_E= \lambda^i \frakD$ where $i=0,1$ by Lemma \ref{semist} Part 1.  We now aim to lift this mod $p$ Bianchi modular form to a complex one.

We compute the abelianizations $\Gamma_0(\frakN_E)^{\rm ab}$ by implementing the algorithm of {\c{S}}eng{\"u}n \cite{Sen11}. One can access to the relevant \texttt{Magma} codes online at \url{https://github.com/ekinozman/ekin}.  The biggest primes $\ell$ that appear as orders of torsion elements of $\Gamma_0(\frakN_E)^{\rm ab}$ are as follows: 

\begin{itemize}
\item $\frakN_E=\frakD$ or $\frakN_E=\lambda \frakD$, $d=7, 19$ : $\ell=3.$
\item $\frakN_E=\frakD$ and $d=43$: $\ell=3.$
\item $\frakN_E=\lambda \frakD$, $d=43$ : $\ell=2531.$
\item $\frakN_E=\frakD$ and $d=67$: $\ell=7.$
\item $\frakN_E=\lambda \frakD$, $d=67$ : $\ell=86338229.$
\end{itemize}

Assume that $\ell>86338229$. It then follows that the $p$-torsion subgroups of $\Gamma_0(\frakN_E)^{\rm ab}$ are all trivial, so the mod $p$ eigenforms must lift to complex ones. The procedure explained at the beginning of Section 5 of \cite{IKO2} together with Conjecture \ref{conj1} imply that there exists a (complex) Bianchi modular form $\frakf$ over $K$ of level $\frakN_E$ such that for all prime ideals $\mathfrak{q}$ coprime to $p\frakN_E$, we have

$${\rm Tr}(\overline{\rho}_{E,p}({\rm Frob}_{\mathfrak{q}})) \equiv \mathfrak{f}(T_{\mathfrak{q}}) \pmod{\mathfrak{p}},$$
where $\frakp$ is a prime ideal of $\mathbb{Q}_\frakf$ lying above $p$ and $\mathbb{Q}_\frakf$ is the number field generated by the eigenvalues. Let us denote this relation by $\overline{\rho}_{E,p}\sim \overline{\rho}_{\frakf,\frakp}$.

By Conjecture~\ref{conj1} and the lifting argument above, we see that  $\overline{\rho}_{E,p}\sim \overline{\rho}_{\frakf,\frakp}$ and the corresponding Bianchi modular form $\frakf$ is of level $ \frakD $ or $\lambda  \frakD$. The final step of the proof is eliminating such forms and getting a contradiction that none of these forms can be associated to a putative solution of the equation.  To achieve this we need the following lemma:

\begin{lemma}\label{ideal_Bf}Let us fix a prime ideal $\mathfrak{q} \nmid \lambda  \frakD$ of $K$, and let $\frakf$ be a newform of level dividing $\lambda  \frakD$. Define the following set 
	$$\mathcal{A}(\frakq)=\{a \in \mathbb{Z}\; :\; |a|\leq 2\sqrt{\Norm(\mathfrak{q})},\; \Norm(\mathfrak{q})+1-a \equiv 0 \pmod 3\}.$$ 
	If $\overline{\rho}_{E,p}\sim \overline{\rho}_{\frakf,\frakp}$, where $\frakp$ is the prime ideal of $\mathbb{Q}_{\frakf}$ lying above $p$, then $\frakp$ divides
	$$B_{\frakf,\mathfrak{q}}:=\Norm(\mathfrak{q})\cdot\left(\Norm(\mathfrak{q} +1)^2-\mathfrak{f}(T_{\mathfrak{q}})^2\right)\cdot\prod_{a \in \mathcal{A}(\frakq)}(a-\mathfrak{f}(T_{\mathfrak{q}}))\mathcal{O}_{\mathbb{Q}_{\frakf}}.$$
\end{lemma}

\begin{proof}
	If $\mathfrak{q}\mid p$, then $\Norm(\mathfrak{q})$ is a power of $p$, hence $\frakp \mid B_{\frakf,\mathfrak{q}} $. 
    
    Now assume that $\mathfrak{q}$ does not divide $p$. Then the Frey curve $E$ has semistable reduction at $\mathfrak{q}$. If it has a good reduction, then we have
	$${\rm Tr}(\overline{\rho}_{E,p}({\rm Frob}_{\mathfrak{q}})) \equiv a_{\mathfrak{q}}(E) \equiv \mathfrak{f}(T_{\mathfrak{q}}) \pmod{\frakp}.$$
	
	Note that the Frey curve $E$ has a $3$-torsion point, so $3$ divides $\#E(\mathbb{F}_{\mathfrak{q}})=\Norm(\mathfrak{q})+1=a_{\mathfrak{q}}(E)$. By Hasse-Weil bound, we know that $|a_{\mathfrak{q}}(E)|\leq 2\sqrt{\Norm(\mathfrak{q})}$. So $a_{\mathfrak{q}}(E)$ belongs to the finite set $\mathcal{A}(\frakq)$. Finally, suppose that $E$ has multiplicative reduction at $\mathfrak{q}$. Then by comparing the traces of the images of Frobenius at $\mathfrak{q}$ under $\overline{\rho}_{E,p}$, we have
	$$\pm(\Norm(\mathfrak{q})+1) \equiv \mathfrak{f}(T_{\mathfrak{q}}) \pmod \frakp.$$
	It then follows that $\frakp$ divides $(\Norm(\mathfrak{q})+1)^2-\mathfrak{f}(T_{\mathfrak{q}})^2$. Hence, $\mathfrak{p}$ divides $B_{\frakf,\mathfrak{q}}$.
\end{proof}

This brings us to the last step of the proof of Theorem 
\ref{thm:sample}.

Using \texttt{Magma}\footnote{see the relevant code and outputs on https://github.com/ekinozman/ekin}, we computed the cuspidal newforms at the predicted levels, the
fields \(\Q_\frakf\), and the eigenvalues \(\frakf(T_{\mathfrak q})\) at prime ideals \(\mathfrak q \in S\) where $S$ is the set of prime ideals $\mathfrak q \neq \lambda$ of norm less
than \(50\) for each imaginary quadratic number field \(K=\mathbb{Q}(\sqrt{-d})\) with
\(d\in\{7,19,43,67\}\).
For each modular form \(\frakf\) of level \(\lambda \frakD\) or of level \(\frakD\)  , we computed the ideal
$ B_\frakf= \Sigma_{\frakq \in S} B_{\frakf, \frakq}.$

Using Lemma \ref{ideal_Bf} we can eliminate the forms as follows:

\begin{itemize}
    \item $K=\Q(\sqrt{-7}):$ There are no Bianchi modular forms of level $ \frakD$ and there is one Bianchi modular form $\mathfrak{f}$ at level $\lambda  \frakD.$ We use Lemma \ref{ideal_Bf} and compute $C_{\frakf}=\Norm_{\Q_\mathfrak{f}/\Q}(B_{\frakf})$ and see that the largest prime divisor of $C_{\frakf}$ is $11$. The bound we get to obtain by Proposition \ref{irrCase2} is $20$. Therefore if $p>20$, the equation $a^p+7b^p=c^3$ has no non-trivial solutions over $\Q(\sqrt{-7})$ such that $(a, db, c)$ is primitive and $\lambda \mid  b.$
    \item $K=\Q(\sqrt{-19}):$ There is one Bianchi modular form $\frakfp$ of level $ \frakD$ and there are three Bianchi modular forms $\frakf_1,\frakf_2,\frakf_3$ at level $\lambda  \frakD.$ We compute $C_{\frakf_i}=\Norm_{\Q_\mathfrak{f}/\Q}(B_{\frakf_i,\mathfrak{q}})$ for $i=1,2,3$ using Lemma \ref{ideal_Bf} and find that the largest prime divisor of $C_{\frakf_i}$ is $7$.  However, $C_\frakfp=0$. Therefore, Lemma \ref{ideal_Bf} is not helpful to eliminate $\frakfp$. To eliminate the form $\frakfp$ we use the following `inertia argument'. Assume that the Frey curve $E$ attached to a putative solution $(a,b,c)$, where $(a, db, c)$ is primitive and $\lambda \mid b$, is associated to the Bianchi modular form $\frakfp$.

 The LMFDB label of the Bianchi modular form $\mathfrak{f'}$ -which is non-trivial and new- is 	$2.0.19.1-19.1-a$. Then Conjecture \ref{conj2}  predicts that there is either an elliptic curve of conductor $ \frakD$ or a fake elliptic curve of conductor $ \frakD^2$ associated to the form $\frakfp$.
 
The elliptic curves in the isogeny classes given in the LMFDB label $2.0.19.1-19.1-a$ correspond to the form $\frakfp$. All the elliptic curves in this isogeny classes have potentially good reduction at $\lambda$. Since the elliptic curve $E$ has potentially multiplicative reduction at the prime $\lambda$, we get a contradiction. To be able to finish the proof we need to eliminate the case that $\frakp$ corresponds to a fake elliptic curve, as it is done in \cite{SS}. 
    Recall  that the Frey curve $E$ has potentially multiplicative reduction at $\lambda$  by Lemma \ref{p divides image of inertia} and $p\nmid v_{\lambda}(j_E)$ as shown in the proof of Corollary \ref{absirr}. 
        Assume that $\mathfrak{f'}$  corresponds to a fake elliptic curve $A_\mathfrak{f'}$. Note that since $\lambda \mid b$,  $p \mid \# \bar{\rho}_{E, p}\left(I_{\lambda}\right)$ by Lemma \ref{p divides image of inertia} for $p > 7$. Since $\mathfrak{f'}$ corresponds to the fake elliptic curve $A_\mathfrak{f'}$, Theorem 4.2 of \cite{SS} yields $\#\bar{\rho}_{A_\mathfrak{f'}, p}\left(I_{\lambda}\right)\le 24$. As $\bar{\rho}_{E, p}\sim \bar{\rho}_{A_\mathfrak{f'}, p}$ and $p> 7$, we have a contradiction. 
        
    Therefore as in the case of $\Q(\sqrt{-7})$, if $p>20$, the equation $a^p+19b^p=c^3$ has no non-trivial solutions over $\Q(\sqrt{-19})$ such that $(a, db, c)$ is primitive and $\lambda \mid  b.$

    \item $K=\Q(\sqrt{-43}):$ There are three Bianchi modular forms of level $ \frakD$ and five Bianchi modular forms of level $\lambda  \frakD.$ Using Lemma \ref{ideal_Bf}, we compute the corresponding $C_{\frakf}$'s and the highest prime divisor we see is $11.$  However, in order to lift the mod $p$ eigenforms to complex ones, we need $p>2531$ as computed in the beginning of Section \ref{effres}.  Therefore if $p>2531$, the equation $a^p+43b^p=c^3$ has no non-trivial solutions over $\Q(\sqrt{-43})$ such that $(a, db, c)$ is primitive and $\lambda \mid  b.$

    \item $K=\Q(\sqrt{-67}):$ There are three Bianchi modular forms of level $ \frakD$ and five  Bianchi modular forms of level $\lambda  \frakD.$ Using Lemma \ref{ideal_Bf}, we compute the corresponding $C_{\frakf}$'s and the highest prime divisor we see is $17$.  However, in order to lift the mod $p$ eigenforms to complex ones, we need $p>86338229$ as computed in the beginning of Section \ref{effres}. Therefore if $p>86338229$, the equation $a^p+67b^p=c^3$ has no non-trivial solutions over $\Q(\sqrt{-67})$ such that $(a, db, c)$ is primitive and $\lambda \mid  b.$ 
\end{itemize}

This finishes the proof of Theorem \ref{thm:sample}.
\bibliographystyle{plain}

\bibliography{reff.bib}

\end{document}